\documentclass[11pt, reqno]{amsart}
\usepackage{amsmath, amsthm, amscd, amsfonts, amssymb, graphicx, color}
\usepackage[bookmarksnumbered, colorlinks, plainpages]{hyperref}

\input{mathrsfs.sty}
\newtheorem{theorem}{Theorem}[section]
\newtheorem{lemma}[theorem]{Lemma}

\newtheorem{corollary}[theorem]{Corollary}
\theoremstyle{definition}

\theoremstyle{remark}
\newtheorem{remark}[theorem]{Remark}
\numberwithin{equation}{section}

\begin{document}
\title[Numerical radius inequalities involving commutators]{Numerical radius
inequalities involving commutators of $G_{1}$ operators}
\author[M. Bakherad and F. Kittaneh ]{Mojtaba Bakherad$^1$ and Fuad Kittaneh$%
^2$}
\address{$^1$Department of Mathematics, Faculty of Mathematics, University
of Sistan and Baluchestan, Zahedan, I.R.Iran.}
\email{mojtaba.bakherad@yahoo.com; bakherad@member.ams.org}
\address{$^2$Department of Mathematics, The University of Jordan, Amman,
Jordan}
\email{fkitt@ju.edu.jo}

\begin{abstract}
We prove numerical radius inequalities involving commutators of $G_{1}$
operators and certain analytic functions. Among other inequalities, it is
shown that if $A$ and $X$ are bounded linear operators on a complex Hilbert
space, then
\begin{equation*}
w(f(A)X+X\bar{f}(A))\leq {\frac{2}{d_{A}^{2}}}w(X-AXA^{\ast }),
\end{equation*}%
where $A$ is a $G_{1}$ operator with $\sigma (A)\subset \mathbb{D}$ and $f$
is analytic on the unit disk $\mathbb{D}$ such that \textrm{{Re}$(f)>0$ and $%
f(0)=1$. }
\end{abstract}

\maketitle



\section{Introduction}

Let $({\mathscr H},\langle \,\cdot \,,\,\cdot \,\rangle )$ be a complex
Hilbert space and ${\mathbb{B}}(\mathscr H)$ denote the $C^{\ast }$-algebra
of all bounded linear operators on ${\mathscr H}$ with the identity $I$. In
the case when {dim}${\mathscr H}=n$, we identify ${\mathbb{B}}({\mathscr H})$
with the matrix algebra $\mathbb{M}_{n}$ of all $n\times n$ matrices having
entries in the complex field. The numerical radius of $A\in {\mathbb{B}}({%
\mathscr H})$ is defined by
\begin{equation*}
w(A):=\sup \Big\{|\langle Ax,x\rangle |:x\in {\mathscr H},\parallel
x\parallel =1\Big\}.
\end{equation*}%
It is well known that $w(\,\cdot \,)$ defines a norm on ${\mathbb{B}}({%
\mathscr H})$, which is equivalent to the usual operator norm $\Vert \,\cdot
\,\Vert $. In fact, for any $A\in {\mathbb{B}}({\mathscr H})$, $\frac{1}{2}%
\Vert A\Vert \leq w(A)\leq \Vert A\Vert $ (see \cite[p. 9]{gof}). If $%
A^{2}=0 $, then equality holds in the first inequality, and if $A$ is
normal, then equality holds in the second inequality. For further
information about numerical radius inequalities, we refer the reader to \cite%
{aA, AK1, AK2, Kitt1, sheikh, YAM} and references therein.

An operator $A\in {\mathbb{B}}({\mathscr H})$ is called a $G_{1}$ operator
if the growth condition
\begin{equation*}
\Vert (z-A)^{-1}\Vert ={\frac{1}{\text{dist}(z,\sigma (A))}}
\end{equation*}%
holds for all $z$ not in the spectrum $\sigma (A)$ of $A$, where $\text{dist}%
(z,\sigma (A))$ denotes the distance between $z$ and $\sigma (A)$. For
simplicity, if $z$ is a complex number, we write $z$ instead of $zI$. It is
known that hyponormal (in particular, normal) operators are $G_{1}$ operators
(see, e.g., \cite{put}). Let $A\in {\mathbb{B}}({\mathscr H})$ and $f$ be a
function which is analytic on an open neighborhood $\Omega $ of $\sigma (A)$
in the complex plane. Then $f(A)$ denotes the operator defined on ${\mathscr %
H}$ by the Riesz-Dunford integral as
\begin{equation*}
f(A)={\frac{1}{2\pi i}}\int_{C}f(z)(z-A)^{-1}dz,
\end{equation*}%
where $C$ is a positively oriented simple closed rectifiable contour
surrounding $\sigma (A)$ in $\Omega $ (see e.g., \cite[p. 568]{du}). The
spectral mapping theorem asserts that $\sigma (f(A))=f(\sigma (A))$.
Throughout this note, $\mathbb{D}=\left\{ z\in \mathbb{C}:|z|<1\right\} $
denotes the unit disk, $\partial \mathbb{D}$ stands for the boundary of $%
\mathbb{D}$ and $d_{A}=\text{dist}(\partial \mathbb{D},\sigma (A))$. In
addition, we denote
\begin{equation*}
\mathfrak{A}=\left\{ f:\mathbb{D}\rightarrow \mathbb{C}:f\,\text{is analytic}%
,\text{Re}(f)>0\,\text{and}\,f(0)=1\right\} .
\end{equation*}%
The Sylvester type equations $AXB\pm X=C$ have been investigated in matrix
theory (see \cite{BL}). Several perturbation bounds for the norms of sums or
differences of operators have been presented in the literature by employing
some integral representations of certain functions. See \cite{Bhatia, kit2,
mosl} and references therein.

In this paper, we present some upper bounds for the numerical radii of the
commutators and elementary operators of the form $f(A)X\pm X\bar{f}(A)$, $%
f(A)X\bar{f}(B)-f(B)X\bar{f}(A)$ and $f(A)X\bar{f}(B)+2X+f(B)X\bar{f}(A)$,
where $A,B,X\in {\mathbb{B}}({\mathscr H})$ and $f\in \mathfrak{A}$.

\section{main results}

To prove our first result, the following lemma concerning numerical radius
inequalities and an equality is required.

\begin{lemma}
\cite{HKSH, HKSH2}\label{kho} Let $A, B, X, Y\in{\mathbb{B}}({\mathscr H})$.
Then\newline

$(a)\,\,w(A^{\ast }XA)\leq \Vert A\Vert ^{2}w(X).$ \newline

$(b)\,\,w\left( AX\pm XA^{\ast }\right) \leq 2\Vert A\Vert w(X).$ \newline

$(c)\,\,w\left( A^{\ast }XB\pm B^{\ast }YA\right) \leq 2\Vert A\Vert \Vert
B\Vert \,w\left( \left[
\begin{array}{cc}
0 & X \\
Y & 0%
\end{array}%
\right] \right) .$\newline

$(d)\,\,w\left( \left[
\begin{array}{cc}
0 & AXB^{\ast } \\
BYA^{\ast } & 0%
\end{array}%
\right] \right) \leq \max \{||A||^{2},||B||^{2}\}w\left( \left[
\begin{array}{cc}
0 & X \\
Y & 0%
\end{array}%
\right] \right) .$ \newline

$(e)\,\,w\left( \left[
\begin{array}{cc}
0 & X \\
Y & 0%
\end{array}%
\right] \right) \leq {\frac{w(X+Y)+w(X-Y)}{2}}.$ \newline

$(f)$ \ $w\left( \left[
\begin{array}{cc}
0 & X \\
e^{i\theta }X & 0%
\end{array}%
\right] \right) =w(X)$ for $\theta \in
\mathbb{R}
$.
\end{lemma}

\begin{proof}
Since all parts, except part (d), have bee shown in \cite{HKSH, HKSH2}, we
prove only part (d). If we take $C=\left[
\begin{array}{cc}
A & 0 \\
0 & B%
\end{array}%
\right] $ and $S=\left[
\begin{array}{cc}
0 & X \\
Y & 0%
\end{array}%
\right] $, then $CSC^{\ast }=\left[
\begin{array}{cc}
0 & AXB^{\ast } \\
BYA^{\ast } & 0%
\end{array}%
\right] $. Now, using part (a), we have
\begin{align*}
w\left( \left[
\begin{array}{cc}
0 & AXB^{\ast } \\
BYA^{\ast } & 0%
\end{array}%
\right] \right) & =w(CSC^{\ast }) \\
& \leq \Vert C\Vert ^{2}w(S) \\
& =\max \{||A||^{2},||B||^{2}\}w\left( \left[
\begin{array}{cc}
0 & X \\
Y & 0%
\end{array}%
\right] \right) \text{,}
\end{align*}%
as required.
\end{proof}

Now, we are in position to demonstrate the main results of this section by
using some ideas from \cite{kit2, mosl}.

\begin{theorem}
\label{kit} Let $A\in {\mathbb{B}}({\mathscr H})$ be a $G_{1}$ operator with
$\sigma (A)\subset \mathbb{D}$ and $f\in \mathfrak{A}$. Then for every $X\in
{\mathbb{B}}({\mathscr H})$, we have%
\begin{equation*}
w(f(A)X+X\bar{f}(A))\leq {\frac{2}{d_{A}^{2}}}w(X-AXA^{\ast })
\end{equation*}%
and
\begin{equation*}
w(f(A)X-X\bar{f}(A))\leq {\frac{4}{d_{A}^{2}}}\Vert A\Vert w(X).
\end{equation*}
\end{theorem}

\begin{proof}
Using the Herglotz representation theorem (see e.g., \cite[p. 21]{do}), we
have%
\begin{equation*}
f(z)=\int_{0}^{2\pi }{\frac{e^{i\alpha }+z}{e^{i\alpha }-z}}d\mu (\alpha )+i%
\text{Im}\,f(0)=\int_{0}^{2\pi }{\frac{e^{i\alpha }+z}{e^{i\alpha }-z}}d\mu
(\alpha ),
\end{equation*}%
where $\mu $ is a positive Borel measure on the interval $[0,2\pi ]$ with
finite total mass $\int_{0}^{2\pi }d\mu (\alpha )=f(0)=1$. Hence,%
\begin{equation*}
\bar{f}({z})=\overline{{\int_{0}^{2\pi }{\frac{e^{i\alpha }+{z}}{e^{i\alpha
}-{z}}}d\mu (\alpha )}}=\int_{0}^{2\pi }{\frac{e^{-i\alpha }+\bar{z}}{%
e^{-i\alpha }-\bar{z}}}d\mu (\alpha ),
\end{equation*}%
where $\bar{f}$ is the conjugate function of $f$. So,%
\begin{align*}
f(A)X+X\bar{f}(A)& =\int_{0}^{2\pi }\left[ \left( e^{i\alpha }+A\right)
\left( e^{i\alpha }-A\right) ^{-1}X+X\left( e^{-i\alpha }+A^{\ast }\right)
\left( e^{-i\alpha }-A^{\ast }\right) ^{-1}\right] d\mu (\alpha ) \\
& =\int_{0}^{2\pi }\left( e^{i\alpha }-A\right) ^{-1}\Big[\left( e^{i\alpha
}+A\right) X\left( e^{-i\alpha }-A^{\ast }\right) \\
& \qquad \quad +\left( e^{i\alpha }-A\right) X\left( e^{-i\alpha }+A^{\ast
}\right) \Big]\left( e^{-i\alpha }-A^{\ast }\right) ^{-1}d\mu (\alpha ) \\
& =2\int_{0}^{2\pi }\left( e^{i\alpha }-A\right) ^{-1}(X-AXA^{\ast })\left(
e^{-i\alpha }-A^{\ast }\right) ^{-1}d\mu (\alpha ).
\end{align*}%
Hence,%
\begin{align*}
w(f(A)X& +X\bar{f}(A)) \\
& =w\left( \int_{0}^{2\pi }\left[ \left( e^{i\alpha }+A\right) \left(
e^{i\alpha }-A\right) ^{-1}X+X\left( e^{-i\alpha }+A^{\ast }\right) \left(
e^{-i\alpha }-A^{\ast }\right) ^{-1}\right] d\mu (\alpha )\right) \\
& =2\,w\left( \int_{0}^{2\pi }\left( e^{i\alpha }-A\right) ^{-1}(X-AXA^{\ast
})\left( e^{-i\alpha }-A^{\ast }\right) ^{-1}d\mu (\alpha )\right) \\
& \leq 2\int_{0}^{2\pi }w\left( \left( e^{i\alpha }-A\right)
^{-1}(X-AXA^{\ast })\left( e^{-i\alpha }-A^{\ast }\right) ^{-1}\right) d\mu
(\alpha ) \\
& \qquad \qquad \qquad \qquad \qquad \qquad (\text{since}\,w(\,\cdot \,)\,%
\text{is a norm}) \\
& \leq 2\int_{0}^{2\pi }\Vert \left( e^{i\alpha }-A\right) ^{-1}\Vert
^{2}w\left( X-AXA^{\ast }\right) d\mu (\alpha ) \\
& \qquad \qquad \qquad \qquad \qquad \qquad (\text{by Lemma}\,\ref{kho}(a)).
\end{align*}%
Since $A$ is a $G_{1}$ operator, it follows that%
\begin{equation*}
\left\Vert \left( e^{i\alpha }-A\right) ^{-1}\right\Vert ={\frac{1}{\text{%
dist}(e^{i\alpha },\sigma (A))}}\leq {\frac{1}{\text{dist}(\partial \mathbb{D%
},\sigma (A))}}={\frac{1}{d_{A}},}
\end{equation*}%
and so%
\begin{align*}
w\left( f(A)X+X\bar{f}(A)\right) & \leq \left( {\frac{2}{d_{A}^{2}}}%
\int_{0}^{2\pi }d\mu (\alpha )\right) w(X-AXA^{\ast }) \\
& =\left( {\frac{2}{d_{A}^{2}}}f(0)\right) w(X-AXA^{\ast }) \\
& ={\frac{2}{d_{A}^{2}}}w(X-AXA^{\ast }).
\end{align*}%
This proves the first inequality.

Similarly, it follows from the equations%
\begin{align*}
f(A)X-& X\bar{f}(A) \\
& =\int_{0}^{2\pi }\left[ \left( e^{i\alpha }+A\right) \left( e^{i\alpha
}-A\right) ^{-1}X-X\left( e^{-i\alpha }+A^{\ast }\right) \left( e^{-i\alpha
}-A^{\ast }\right) ^{-1}\right] d\mu (\alpha ) \\
& =\int_{0}^{2\pi }\left( e^{i\alpha }-A\right) ^{-1}\Big[\left( e^{i\alpha
}+A\right) X\left( e^{-i\alpha }-A^{\ast }\right) \\
& \qquad \quad -\left( e^{i\alpha }-A\right) X\left( e^{-i\alpha }+A^{\ast
}\right) \Big]\left( e^{-i\alpha }-A^{\ast }\right) ^{-1}d\mu (\alpha ) \\
& =2\int_{0}^{2\pi }\left( e^{i\alpha }-A\right) ^{-1}(e^{-i\alpha
}AX-e^{i\alpha }XA^{\ast })\left( e^{-i\alpha }-A^{\ast }\right) ^{-1}d\mu
(\alpha ) \\
& =2\int_{0}^{2\pi }\left( e^{i\alpha }-A\right) ^{-1}\left( \left(
e^{-i\alpha }A\right) X-X\left( e^{-i\alpha }A\right) ^{\ast }\right) \left(
e^{-i\alpha }-A^{\ast }\right) ^{-1}d\mu (\alpha )
\end{align*}%
that
\begin{align*}
w(& f(A)X-X\bar{f}(A)) \\
& =2w\left( \int_{0}^{2\pi }\left( e^{i\alpha }-A\right) ^{-1}\left( \left(
e^{-i\alpha }A\right) X-X\left( e^{-i\alpha }A\right) ^{\ast }\right) \left(
e^{-i\alpha }-A^{\ast }\right) ^{-1}d\mu (\alpha )\right) \\
& \leq 2\int_{0}^{2\pi }w\left( \left( e^{i\alpha }-A\right) ^{-1}\left(
\left( e^{-i\alpha }A\right) X-X\left( e^{-i\alpha }A\right) ^{\ast }\right)
\left( e^{-i\alpha }-A^{\ast }\right) ^{-1}\right) d\mu (\alpha ) \\
& \qquad \qquad \qquad \qquad \qquad \qquad (\text{since}\,w(\,\cdot \,)\,%
\text{is a norm}) \\
& \leq 2\int_{0}^{2\pi }\left\Vert \left( e^{i\alpha }-A\right)
^{-1}\right\Vert ^{2}w\left( \left( e^{-i\alpha }A\right) X-X\left(
e^{-i\alpha }A\right) ^{\ast }\right) d\mu (\alpha ) \\
& \qquad \qquad \qquad \qquad \qquad \qquad (\text{by Lemma\thinspace \ref%
{kho} (a)}) \\
& \leq 4\int_{0}^{2\pi }\left\Vert \left( e^{i\alpha }-A\right)
^{-1}\right\Vert ^{2}\Vert e^{-i\alpha }A\Vert w(X)d\mu (\alpha ) \\
& \qquad \qquad \qquad \qquad \qquad \qquad (\text{by Lemma\thinspace \ref%
{kho} (b)}) \\
& \leq \frac{4}{d_{A}^{2}}\Vert A\Vert w(X)\int_{0}^{2\pi }d\mu (\alpha ) \\
& \leq \frac{4}{d_{A}^{2}}\Vert A\Vert w(X).
\end{align*}%
This proves the second inequality and completes the proof of the theorem.
\end{proof}

If we take $X=I$ in Theorem \ref{kit}, we get the following result. Observe
that $\bar{f}(A)=\left( f(A)\right) ^{\ast }$.

\begin{corollary}
Let $A\in{\mathbb{B}}({\mathscr H})$ be a $G_1$ operator with $\sigma
(A)\subset\mathbb{D}$ and $f\in\mathfrak{A}$. Then
\begin{align*}
\|\text{Re}(f(A))\|\leq{\frac{1}{d^2_A}}\|I-AA^*\|
\end{align*}
and
\begin{align*}
\|\text{Im} (f(A))\|\leq{\frac{2}{d^2_A}}\|A\|.
\end{align*}
\end{corollary}

\begin{theorem}
\label{++} Let $A,B\in {\mathbb{B}}({\mathscr H})$ be $G_{1}$ operators with
$\sigma (A)\cup \sigma (B)\subset \mathbb{D}$ and $f\in \mathfrak{A}$. Then
for every $X\in {\mathbb{B}}({\mathscr H})$, we have%
\begin{align*}
w(f(A)X\bar{f}(B)& -f(B)X\bar{f}(A)) \\
& \leq {\frac{2}{d_{A}d_{B}}}\,\left[ 2w\left( X\right) +w\left( AXB^{\ast
}+BXA^{\ast }\right) +w\left( AXB^{\ast }-BXA^{\ast }\right) \right]
\end{align*}%
and
\begin{align*}
w(f(A)X\bar{f}(B)& +2X+f(B)X\bar{f}(A)) \\
& \leq {\frac{2}{d_{A}d_{B}}}\,\left[ 2w\left( X\right) +w\left( AXB^{\ast
}+BXA^{\ast }\right) +w\left( AXB^{\ast }-BXA^{\ast }\right) \right] .
\end{align*}
\end{theorem}

\begin{proof}
We have%
\begin{align*}
f(A)X\bar{f}(B)& -f(B)X\bar{f}(A) \\
& =\int_{0}^{2\pi }\int_{0}^{2\pi }\Big[\left( e^{i\alpha }-A\right)
^{-1}(e^{i\alpha }+A)X(e^{-i\beta }+B^{\ast })\left( e^{-i\beta }-B^{\ast
}\right) ^{-1} \\
& \,\,\,\,-\left( e^{i\beta }-B\right) ^{-1}(e^{i\beta }+B)X(e^{-i\alpha
}+A^{\ast })\left( e^{-i\alpha }-A^{\ast }\right) ^{-1}\Big]d\mu (\alpha
)d\mu (\beta ).
\end{align*}%
Using the equations%
\begin{align*}
& \left( e^{i\alpha }-A\right) ^{-1}(e^{i\alpha }+A)X(e^{-i\beta }+B^{\ast
})\left( e^{-i\beta }-B^{\ast }\right) ^{-1} \\
& \,\,\,\,-\left( e^{i\beta }-B\right) ^{-1}(e^{i\beta }+B)X(e^{-i\alpha
}+A^{\ast })\left( e^{-i\alpha }-A^{\ast }\right) ^{-1} \\
& =\left( e^{i\alpha }-A\right) ^{-1}(e^{i\alpha }+A)X(e^{-i\beta }+B^{\ast
})\left( e^{-i\beta }-B^{\ast }\right) ^{-1}+X \\
& \,\,\,\,-X-\left( e^{i\beta }-B\right) ^{-1}(e^{i\beta }+B)X(e^{-i\beta
}+A^{\ast })\left( e^{-i\alpha }-A^{\ast }\right) ^{-1} \\
& =\left( e^{i\alpha }-A\right) ^{-1}\left[ (e^{i\alpha }+A)X(e^{-i\beta
}+B^{\ast })+(e^{i\alpha }-A)X(e^{-i\beta }-B^{\ast })\right] \left(
e^{-i\beta }-B^{\ast }\right) ^{-1} \\
& \,\,-\left( e^{i\beta }-B\right) ^{-1}\left[ (e^{i\beta }-B)X(e^{-i\alpha
}-A^{\ast })+(e^{i\beta }+B)X(e^{-i\alpha }+A^{\ast })\right] \left(
e^{-i\alpha }-A^{\ast }\right) ^{-1} \\
& =2(e^{i\alpha }-A)^{-1}(e^{i\alpha }e^{-i\beta }X+AXB^{\ast })(e^{-i\beta
}-B^{\ast })^{-1} \\
& \,\,-2(e^{i\beta }-B)^{-1}(e^{-i\alpha }e^{i\beta }X+BXA^{\ast
})(e^{-i\alpha }-A^{\ast })^{-1},
\end{align*}%
we have%
\begin{align*}
w(f& (A)X\bar{f}(B)-f(B)X\bar{f}(A)) \\
& =2w\Big(\int_{0}^{2\pi }\int_{0}^{2\pi }(e^{i\alpha }-A)^{-1}(e^{i\alpha
}e^{-i\beta }X+AXB^{\ast })(e^{-i\beta }-B^{\ast })^{-1} \\
& \qquad -(e^{i\beta }-B)^{-1}(e^{-i\alpha }e^{i\beta }X+BXA^{\ast
})(e^{-i\alpha }-A^{\ast })^{-1}d\mu (\alpha )d\mu (\beta )\Big) \\
& \leq 2\int_{0}^{2\pi }\int_{0}^{2\pi }w\Big((e^{i\alpha
}-A)^{-1}(e^{i\alpha }e^{-i\beta }X+AXB^{\ast })(e^{-i\beta }-B^{\ast })^{-1}
\\
& \qquad -(e^{i\beta }-B)^{-1}(e^{-i\alpha }e^{i\beta }X+BXA^{\ast
})(e^{-i\alpha }-A^{\ast })^{-1}\Big)d\mu (\alpha )d\mu (\beta ) \\
& \qquad \qquad \qquad \qquad \qquad \qquad (\text{since}\,w(\,\cdot \,)\,%
\text{is a norm})
\end{align*}

\begin{align*}
& \leq 4\int_{0}^{2\pi }\int_{0}^{2\pi }\Vert (e^{i\alpha }-A)^{-1}\Vert
\Vert (e^{i\beta }-B)^{-1}\Vert \\
& \qquad \times w\left( \left[
\begin{array}{cc}
0 & e^{i\alpha }e^{-i\beta }X+AXB^{\ast } \\
e^{-i\alpha }e^{i\beta }X+BXA^{\ast } & 0%
\end{array}%
\right] \right) d\mu (\alpha )d\mu (\beta ) \\
& \qquad \qquad \qquad \qquad \qquad (\text{by Lemma \thinspace \ref{kho} (c)%
}) \\
& \leq {\frac{4}{d_{A}d_{B}}}\int_{0}^{2\pi }\int_{0}^{2\pi }\left[ w\left( %
\left[
\begin{array}{cc}
0 & e^{i\alpha }e^{-i\beta }X \\
e^{-i\alpha }e^{i\beta }X & 0%
\end{array}%
\right] \right) \right. \\
& \qquad \left. +w\left( \left[
\begin{array}{cc}
0 & AXB^{\ast } \\
BXA^{\ast } & 0%
\end{array}%
\right] \right) \right] d\mu (\alpha )d\mu (\beta ) \\
& ={\frac{4}{d_{A}d_{B}}}\int_{0}^{2\pi }\int_{0}^{2\pi }\left[ w\left( %
\left[
\begin{array}{cc}
0 & X \\
X & 0%
\end{array}%
\right] \right) +w\left( \left[
\begin{array}{cc}
0 & AXB^{\ast } \\
BXA^{\ast } & 0%
\end{array}%
\right] \right) \right] d\mu (\alpha )d\mu (\beta ) \\
& \leq {\frac{2}{d_{A}d_{B}}}\left[ 2w\left( X\right) +w\left( AXB^{\ast
}+BXA^{\ast }\right) +w\left( AXB^{\ast }-BXA^{\ast }\right) \right] \\
& \qquad \qquad \qquad \qquad \qquad (\text{by Lemma \thinspace \ref{kho}
(e) and (f)}).
\end{align*}%
This proves the first inequality.

Similarly, we have%
\begin{align*}
f(A)X\bar{f}(B)& +2X+f(B)X\bar{f}(A) \\
& =\int_{0}^{2\pi }\int_{0}^{2\pi }\Big[\left( e^{i\alpha }-A\right)
^{-1}(e^{i\alpha }+A)X(e^{-i\beta }+B^{\ast })\left( e^{-i\beta }-B^{\ast
}\right) ^{-1}+2X \\
& \,\,\,\,+\left( e^{i\beta }-B\right) ^{-1}(e^{i\beta }+B)X(e^{-i\alpha
}+A^{\ast })\left( e^{-i\alpha }-A^{\ast }\right) ^{-1}\Big]d\mu (\alpha
)d\mu (\beta ).
\end{align*}%
Using the equations%
\begin{align*}
& \left( e^{i\alpha }-A\right) ^{-1}(e^{i\alpha }+A)X(e^{-i\beta }+B^{\ast
})\left( e^{-i\beta }-B^{\ast }\right) ^{-1}+2X \\
& \,\,\,\,+\left( e^{i\beta }-B\right) ^{-1}(e^{i\beta }+B)X(e^{-i\alpha
}+A^{\ast })\left( e^{-i\alpha }-A^{\ast }\right) ^{-1} \\
& =\left( e^{i\alpha }-A\right) ^{-1}(e^{i\alpha }+A)X(e^{-i\beta }+B^{\ast
})\left( e^{-i\beta }-B^{\ast }\right) ^{-1}+X \\
& \,\,\,\,+X+\left( e^{i\beta }-B\right) ^{-1}(e^{i\beta }+B)X(e^{-i\beta
}+A^{\ast })\left( e^{-i\alpha }-A^{\ast }\right) ^{-1} \\
& =\left( e^{i\alpha }-A\right) ^{-1}\left[ (e^{i\alpha }+A)X(e^{-i\beta
}+B^{\ast })+(e^{i\alpha }-A)X(e^{-i\beta }-B^{\ast })\right] \left(
e^{-i\beta }-B^{\ast }\right) ^{-1} \\
& \,\,+\left( e^{i\beta }-B\right) ^{-1}\left[ (e^{i\beta }-B)X(e^{-i\alpha
}-A^{\ast })+(e^{i\beta }+B)X(e^{-i\alpha }+A^{\ast })\right] \left(
e^{-i\alpha }-A^{\ast }\right) ^{-1} \\
& =2(e^{i\alpha }-A)^{-1}(e^{i\alpha }e^{-i\beta }X+AXB^{\ast })(e^{-i\beta
}-B^{\ast })^{-1} \\
& \,\,+2(e^{i\beta }-B)^{-1}(e^{-i\alpha }e^{i\beta }X+BXA^{\ast
})(e^{-i\alpha }-A^{\ast })^{-1},
\end{align*}%
we have%
\begin{align*}
w(f& (A)X\bar{f}(B)+2X+f(B)X\bar{f}(A)) \\
& =2w\Big(\int_{0}^{2\pi }\int_{0}^{2\pi }(e^{i\alpha }-A)^{-1}(e^{i\alpha
}e^{-i\beta }X+AXB^{\ast })(e^{-i\beta }-B^{\ast })^{-1} \\
& \qquad +(e^{i\beta }-B)^{-1}(e^{-i\alpha }e^{i\beta }X+BXA^{\ast
})(e^{-i\alpha }-A^{\ast })^{-1}d\mu (\alpha )d\mu (\beta )\Big) \\
& \leq 2\int_{0}^{2\pi }\int_{0}^{2\pi }w\Big((e^{i\alpha
}-A)^{-1}(e^{i\alpha }e^{-i\beta }X+AXB^{\ast })(e^{-i\beta }-B^{\ast })^{-1}
\\
& \qquad +(e^{i\beta }-B)^{-1}(e^{-i\alpha }e^{i\beta }X+BXA^{\ast
})(e^{-i\alpha }-A^{\ast })^{-1}\Big)d\mu (\alpha )d\mu (\beta ) \\
& \qquad \qquad \qquad \qquad \qquad \qquad (\text{since}\,w(\,\cdot \,)\,%
\text{is a norm}) \\
& \leq 4\int_{0}^{2\pi }\int_{0}^{2\pi }\Vert (e^{i\alpha }-A)^{-1}\Vert
\Vert (e^{i\beta }-B)^{-1}\Vert \\
& \qquad \times w\left( \left[
\begin{array}{cc}
0 & e^{i\alpha }e^{-i\beta }X+AXB^{\ast } \\
e^{-i\alpha }e^{i\beta }X+BXA^{\ast } & 0%
\end{array}%
\right] \right) d\mu (\alpha )d\mu (\beta ) \\
& \qquad \qquad \qquad \qquad \qquad (\text{by Lemma \thinspace \ref{kho} (c)%
}) \\
& \leq {\frac{4}{d_{A}d_{B}}}\int_{0}^{2\pi }\int_{0}^{2\pi }\left[ w\left( %
\left[
\begin{array}{cc}
0 & e^{i\alpha }e^{-i\beta }X \\
e^{-i\alpha }e^{i\beta }X & 0%
\end{array}%
\right] \right) \right. \\
& \qquad \left. +w\left( \left[
\begin{array}{cc}
0 & AXB^{\ast } \\
BXA^{\ast } & 0%
\end{array}%
\right] \right) \right] d\mu (\alpha )d\mu (\beta ) \\
& ={\frac{4}{d_{A}d_{B}}}\int_{0}^{2\pi }\int_{0}^{2\pi }\left[ w\left( %
\left[
\begin{array}{cc}
0 & X \\
X & 0%
\end{array}%
\right] \right) +w\left( \left[
\begin{array}{cc}
0 & AXB^{\ast } \\
BXA^{\ast } & 0%
\end{array}%
\right] \right) \right] d\mu (\alpha )d\mu (\beta ) \\
& \leq {\frac{2}{d_{A}d_{B}}}\left[ 2w\left( X\right) +w\left( AXB^{\ast
}+BXA^{\ast }\right) +w\left( AXB^{\ast }-BXA^{\ast }\right) \right] \\
& \qquad \qquad \qquad \qquad \qquad (\text{by Lemma \thinspace \ref{kho}
(e) and (f)}).
\end{align*}%
This proves the second inequality and completes the proof of the theorem.
\end{proof}

\begin{remark}
Under the assumptions of Theorem \ref{++} and the hypothesis that $X$ is
self-adjoint, we have
\begin{align*}
\Vert f(A)X\bar{f}(B)& -f(B)X\bar{f}(A)\Vert \\
& \leq {\frac{4}{d_{A}d_{B}}}\,\max \{\Vert \,|X|\,\Vert +\Vert\, |AXB^{\ast
}|\,\Vert ,\Vert \,|X|\,\Vert +\Vert \,|BXA^{\ast }|\,\Vert \}
\end{align*}%
and
\begin{align*}
\Vert f(A)X\bar{f}(B)& +2X+f(B)X\bar{f}(A)\Vert \\
& \leq {\frac{4}{d_{A}d_{B}}}\,\max \{\Vert \,|X|\,\Vert +\Vert \,|AXB^{\ast
}|\,\Vert ,\Vert\, |X|\,\Vert +\Vert \,|BXA^{\ast }|\,\Vert \}.
\end{align*}%
To see this, first note that if $X$ is self-adjoint, then the operator
matrix
\begin{equation*}
T=\left[
\begin{array}{cc}
0 & e^{i\alpha }e^{-i\beta }X+AXB^{\ast } \\
e^{-i\alpha }e^{i\beta }X+BXA^{\ast } & 0%
\end{array}%
\right]
\end{equation*}%
is self-adjoint, whence $w(T)=\Vert T\Vert $. Moreover, $T=M+N$, where
\begin{equation*}
M=\left[
\begin{array}{cc}
0 & e^{i\alpha }e^{-i\beta }X \\
e^{-i\alpha }e^{i\beta }X & 0%
\end{array}%
\right] ,\qquad N=\left[
\begin{array}{cc}
0 & AXB^{\ast } \\
BXA^{\ast } & 0%
\end{array}%
\right]
\end{equation*}%
are self-adjoint operators. Using the fact that $\Vert C+D\Vert \leq \Vert\,
|C|+|D|\,\Vert $ for any normal operators $C$ and $D$ (see \cite{Bour}), we
have
\begin{equation*}
w(T)=\Vert M+N\Vert \leq \Vert \,|M|+|N|\,\Vert =\max \{\Vert \,|X|\,\Vert +\Vert
\,|AXB^{\ast }|\,\Vert ,\Vert \,|X|\,\Vert +\Vert \,|BXA^{\ast }|\,\Vert \}.
\end{equation*}%
Hence, we get the required inequalities by the same arguments as in the
proof of Theorem \ref{++}.
\end{remark}

If we take $X=I$ in Theorem \ref{++}, we get the following result.

\begin{corollary}
\label{im} Let $A, B\in{\mathbb{B}}({\mathscr H})$ be $G_1$ operators with $%
\sigma (A)\cup\sigma(B)\subset\mathbb{D}$ and $f\in\mathfrak{A}$. Then
\begin{align*}
\|\text{Im}(f(A)\bar{f}(B))\|\leq {\frac{2}{d_Ad_B}}\,\left(1+\|AB^*\|\right)
\end{align*}
and
\begin{align*}
\|\text{Re}(f(A)\bar{f}(B))+I\|\leq {\frac{2}{d_Ad_B}}\,\left(1+\|AB^*\|%
\right).
\end{align*}
\end{corollary}

\begin{remark}
If instead of applying Lemma \ref{kho} (c) we use Lemma \ref{kho} (d) and
(f) in the proof Theorem \ref{++}, we obtain the related inequalities%
\begin{equation*}
w(f(A)X\bar{f}(B)-f(B)X\bar{f}(A))\leq {\frac{4}{d_{A}d_{B}}}\,\left[ 1+\max
\{\Vert A\Vert ^{2},\Vert B\Vert ^{2}\}\right] w\left( X\right)
\end{equation*}%
and
\begin{equation*}
w(f(A)X\bar{f}(B)+2X+f(B)X\bar{f}(A))\leq {\frac{4}{d_{A}d_{B}}}\,\left[
1+\max \{\Vert A\Vert ^{2},\Vert B\Vert ^{2}\}\right] w\left( X\right) .
\end{equation*}
\end{remark}
\textbf{Acknowledgement.} The first author would like to thank the Tusi Mathematical Research Group (TMRG).
\bigskip

\end{document}